\documentclass[12pt,reqno]{amsart}

\usepackage{amsmath}
\usepackage{amssymb}
\usepackage{graphicx}
\usepackage{mathrsfs}
\usepackage{color}
\usepackage{comment}
\usepackage{hyperref}
\usepackage{enumitem}
\usepackage{bm}
\usepackage[margin=1.25in]{geometry}
\usepackage{mnotes}

\numberwithin{equation}{section}

\newtheorem{prop}{Proposition}[section]
\newtheorem{theo}[prop]{Theorem}
\newtheorem*{theo*}{Theorem}
\newtheorem{lemm}[prop]{Lemma}
\newtheorem{coro}[prop]{Corollary}

\newtheorem{claim}{Claim}
\newtheorem{defi}[prop]{Definition}

\theoremstyle{definition}
\newtheorem{rema}[prop]{Remark}

\newtheorem{quest}[prop]{Question}

\newcommand{\PP}{\mathbf{P}}

\newcommand{\RR}{\mathbf{R}}
\renewcommand{\SS}{\mathbf{S}}

\newcommand{\ZZ}{\mathbf{Z}}

\newcommand{\sM}{\mathscr{M}}

\DeclareMathOperator{\imag}{im}

\DeclareMathOperator{\rank}{rank}

\newcommand{\eps}{\varepsilon}

\newcommand\quot[2]{
	\mathchoice
	{
		\text{\raise1ex\hbox{$#1$}\Big/\lower1ex\hbox{$#2$}}%
	}
	{
		#1\,/\,#2
	}
	{
		#1\,/\,#2
	}
	{
		#1\,/\,#2
	}
}

\numberwithin{equation}{section}

\begin{document}

\title{Decomposing 4-manifolds with positive scalar curvature}

\author{Richard H. Bamler}
\address{Department of Mathematics, University of California, Berkeley, Berkeley, CA 94720, USA}
\email{rbamler@berkeley.edu}

\author{Chao Li}
\address{Courant Institute, New York University, 251 Mercer St, New York, NY 10012, USA}
\email{chaoli@nyu.edu}

\author{Christos Mantoulidis}
\address{Department of Mathematics, Rice University, Houston, TX 77005, USA}
\email{christos.mantoulidis@rice.edu}

\begin{abstract}
We show that every closed, oriented, topologically PSC 4-manifold can be obtained via 0 and 1-surgeries from a topologically PSC 4-orbifold with vanishing first Betti number and second Betti number at most as large as the original one. 
\end{abstract}

\maketitle

\section{Introduction}

Our starting point is the following well-known theorem of Schoen--Yau and Gromov--Lawson, which exhibits the richness of the class of manifolds that can carry Riemannian metrics with positive scalar curvature. For convenience and brevity, we will adopt the convention of writing ``PSC'' for ``positive scalar curvature,'' and will call a manifold  ``topologically PSC'' if it can carry Riemannian metrics with positive scalar curvature.

\begin{theo}[{\cite{SchoenYau:structure.psc, GromovLawson:classification}}]  \label{theo:sy.gl.surgery}
	Let $M$ be a topologically PSC $n$-manifold with $n \geq 3$. Any manifold obtained from $M$ by performing a sequence of $0$-, $1$-, ..., and/or $(n-3)$-surgeries is also topologically PSC.
\end{theo}

Theorem \ref{theo:sy.gl.surgery} naturally leads one to ask:

\begin{quest} \label{quest:sy.gl.converse}
	Can all topologically PSC $n$-manifolds be built out of ``simple'' topologically PSC $n$-manifolds by performing codimension $\geq 3$ surgeries?
\end{quest} 

Our understanding of $3$-manifold topology implies a positive answer when $n=3$:
\begin{theo}[Schoen--Yau and Gromov--Lawson and Perelman] \label{theo:full.3d}
	Every closed, oriented, topologically PSC $3$-manifold can be obtained by performing $0$-surgeries on a disjoint union of spherical space forms (i.e., $\SS^3/\Gamma$'s, where the $\Gamma$'s are finite subgroups of $SO(4)$ acting freely on $\SS^3$).
\end{theo}

In this paper we prove:

\begin{theo} \label{theo:homology.4d}
	Every closed, oriented, topologically PSC $4$-manifold $M$ can be obtained from a possibly disconnected, closed, oriented, topologically PSC $4$-orbifold $M'$ with isolated singularities such that $b_1(M') = 0$ and $b_2(M') \leq b_2(M)$ by performing $0$- and $1$-surgeries. All $1$-surgeries are standard manifold ones, but $0$-surgeries may occur at orbifold points. 
\end{theo}

Recall that the $j$th Betti number $b_j(M')$ of the orbifold $M'$ is defined to be the $j$th Betti number of $M'$ viewed as a topological space; 
in our case, this is equivalent to the $j$th Betti number of the regular part $M'_{\text{reg}} \subset M'$. In the connected case, $b_1(M')$ is the same as the rank of the abelianization of the orbifold fundamental group $\pi_1^{\text{orb}}(M')$. Finally, a $0$-surgery occurring at two orbifold points both modeled on $\RR^4/\Gamma$ means that the corresponding connected sum operation is performed with a $\SS^3/\Gamma$ neck. 

\begin{rema} \label{rema:homology.4d.orbifold}
	Theorem \ref{theo:homology.4d} will continue to hold if $M$ is itself a $4$-orbifold rather than a $4$-manifold, but this is somewhat outside the scope of our current paper.
\end{rema}

Note that orbifolds are indeed sometimes necessary for surgery decompositions such as that of Theorem \ref{theo:homology.4d} in dimension $4$. 

\begin{theo} \label{theo:s3g.s1}
	Let $\SS^3/\Gamma$ be a lens space with $\Gamma$ a nontrivial finite cyclic subgroup of $SO(4)$. The $4$-manifold $M = (\SS^3/\Gamma) \times \SS^1$  \textit{cannot} be obtained by performing manifold $0$ and $1$-surgeries on a $4$-manifold $M'$ with $b_1(M') = 0$.
\end{theo}

Let us outline the proof of Theorem \ref{theo:homology.4d}. Endow $M$ with an arbitrary PSC metric. We ``exhaust'' the codimension-$1$ homology of $M$ with a two-sided, stable minimal hypersurface $\Sigma$ (Lemma \ref{lemm:homology.nd.stable}). By a now-standard argument of Schoen--Yau, the metric induced on $\Sigma$ is conformal to a PSC metric. Thus, $\Sigma$ is topologically the result of $0$-surgeries on spherical space forms (Theorem \ref{theo:full.3d}). We then show how to locally modify the metric on $M$ to another PSC metric that is locally a product near $\Sigma$ and induces a ``model'' PSC metric on $\Sigma$ (Lemma \ref{lemm:3d.metric.modification}, Definition \ref{defi:3d.model.metric}). If $\Sigma$ is merely the disjoint union of spherical space forms $\SS^3/\Gamma$, with no $0$-surgeries, then our model metrics are all round and simple $3$-surgeries on $M$ along the components of $\Sigma$ yield a $4$-orbifold whose $b_2$ is unchanged and $b_1$ is trivial (because $\Sigma$ suitably exhausted the codimension-$1$ homology of $M$). If $\Sigma$ does involve $0$-surgeries, we first undo these using $2$-surgeries on $M$ near $\Sigma$'s $0$-surgery neck regions; this may decrease $b_2$. We have only performed $3$- and $2$-surgeries on $M$ to get to the orbifold, so $M$ can be obtained from the orbifold via $0$- and $1$-surgeries, respectively. Note that our construction preserves the spin condition; see Remark \ref{rema:spin.4d.conclusion}.

We conclude our introduction by posing the following:
\begin{quest}
	Let $M$ be a closed, oriented, topologically PSC $4$-manifold. Can one obtain $M$ from a closed, oriented, topologically PSC $4$-orbifold $M'$ with isolated singularities and the property that each component has finite orbifold fundamental group $\pi_1^{\text{orb}}(M')$ by performing $0$- and $1$-surgeries?
\end{quest}

\textit{Organization of the paper}. We first illustrate our homological decomposition argument for ambient $3$-manifolds in Section \ref{sec:homology.3d}, where the picture is much simpler. Section \ref{sec:homotopy.3d} is a brief digression offering a homotopical refinement of the $3$-dimensional decomposition. In Section \ref{sec:homology.4d}, we give the proof of our main theorem. In Section \ref{sec:s3g.s1} we give the proof of Theorem \ref{theo:s3g.s1}. Finally, our appendix contains the proofs of some technical lemmas.

\medskip

\textit{Acknowledgments}. R.B. was partially supported by NSF grant DMS-1906500. C.L. was partially supported by NSF grant DMS-2202343. C.M. was partially supported by NSF Grant DMS-2147521. The authors would like to thank the anonymous referee for their suggestions and Otis Chodosh for some helpful conversations.

\medskip

\section{Warmup: homology decomposition in 3D} \label{sec:homology.3d}

We first illustrate our homological decomposition argument on 3D manifolds, where it yields the following weaker analog of Theorem \ref{theo:full.3d}. (See, however, Theorem \ref{theo:homotopy.3d} and its subsequent discussion.)

\begin{theo} \label{theo:homology.3d}
	Every closed, oriented, topologically PSC $3$-manifold $M$ can be obtained from a closed, oriented, topologically PSC $3$-manifold $M'$ with $b_1(M') = 0$ by performing $0$-surgeries.
\end{theo}

\begin{rema}[The idea in 3D] \label{rema:idea.3d}
We want to capture $H_2(M; \ZZ)$ by $2$-spheres on which we can perform $2$-surgeries within the PSC category. Of course, $2$-surgeries in a $3$-manifold do not a priori preserve the PSC condition (they are not of codimension $\geq 3$), but we will see that $2$-surgeries performed on stable minimal spheres do. Indeed, we show how to reduce a  $2$-surgery on any stable minimal sphere to the surgically trivial case of $2$-surgery on a round, totally geodesic sphere. There, the PSC condition is obviously preserved under the surgery. 
\end{rema}

Lemma \ref{lemm:2d.metric.modification} is the preparation lemma underlying our PSC $2$-surgeries. We defer its proof to Appendix \ref{sec:technical.lemmas}.

\begin{lemm}[Metric preparation lemma] \label{lemm:2d.metric.modification}
	Let $\Sigma$ be a closed, embedded, two-sided, stable minimal surface in an oriented, PSC $3$-manifold $(M, g)$. Then:
	\begin{enumerate}
		\item[(a)] Each component of $\Sigma$ is an $\SS^2$.
		\item[(b)] Given any auxiliary PSC metric $\varrho$ on $\Sigma$, there exists a new PSC metric $\tilde g$ on $M$, which:
			\begin{itemize}
				\item is isometric to a product cylinder $(\Sigma, \varrho) \times (-2, 2)$ in the distance-$2$ tubular neighborhood of $\Sigma$, and
				\item coincides with $g$ outside a larger tubular neighborhood of $\Sigma$.
			\end{itemize}
	\end{enumerate}
\end{lemm}

\begin{rema} \label{rema:2d.metric.modification.onesided}
	If $\Sigma$ has one-sided components, then the same is true for them, except they are $\RR \PP^2$'s rather than $\SS^2$'s, and around them each product cylinder $(\SS^2, \varrho) \times (-2, 2)$ is to be taken modulo the $\ZZ_2$ action $(x, t) \mapsto (-x, -t)$; here, the auxiliary metric $\varrho$ must be a lift from a metric on $\RR \PP^2$ to a stable two-sided $\SS^2$ covering.
\end{rema}

The lemma below delivers the surface $\Sigma$ on which Lemma \ref{lemm:2d.metric.modification} is applied. We state it in $n$-dimensional generality, because it is indeed very general, and we will it need again in Section \ref{sec:homology.4d}. We defer its proof to Appendix \ref{sec:technical.lemmas}, and note that it can be viewed as a generalization of the ``slicing'' procedure in \cite[Lemma 19]{ChodoshLi:aspherical}.

\begin{lemm}[Minimal (hyper)surface preparation lemma] \label{lemm:homology.nd.stable}
	Let $(M, g)$ be a closed, connected, oriented $n$-manifold, with $n \leq 7$. There exists a closed, embedded, two-sided, stable minimal hypersurface $\Sigma \subset M$ such that $\breve M = M \setminus \Sigma$ is connected and the map $H_{n-1}(\partial \breve M; \ZZ) \twoheadrightarrow H_{n-1}(\breve M; \ZZ)$ is surjective.
\end{lemm}

\begin{proof}[Proof of Theorem \ref{theo:homology.3d}]
	Without loss of generality, $M$ is connected. Endow $M$ with a PSC metric $g$. Let $\Sigma$ be as in Lemma \ref{lemm:homology.nd.stable}, with components $\{ \Sigma_i \}_{i=1}^k$. By Lemma \ref{lemm:2d.metric.modification} and the two-sidedness of each component $\Sigma_i$, and we can modify the metric $g$ near $\Sigma_i$ and arrange for a new PSC metric $\tilde g$ on $M$ inducing a product metric isometric to $(\SS^2, \varrho_{\SS^2}) \times (-2, 2)$ on the distance-$2$ tubular neighborhood of $\Sigma_i$, where $\varrho_{\SS^2}$ is the standard round metric. 
	
	Now excise the distance-$1$ tubular neighborhood $U_i$ of each $\Sigma_i$ and cap off the two newly formed boundary components with a pair of PSC $3$-balls $\{ (B^\alpha_i, \eta^\alpha_i) \}_{\alpha=1}^2$ whose metric has been suitably deformed near the boundary to match smoothly with a product cylinder over $(\SS^2, \varrho_{\SS^2})$. Then,
	\[ (M', g') := (M \setminus \cup_{i=1}^k U_i, \tilde g) \cup \big( \cup_{i=1}^k \cup_{\alpha=1,2} (B_i^\alpha, \eta^\alpha_i) \big) \]
	with the obvious boundary identifications is a PSC $3$-manifold obtained from $M$ by performing a $2$-surgery for each $\Sigma_i$. Thus, $M$ can be obtained from $M'$ by performing $0$-surgeries.
	
	It remains to verify $b_1(M') = 0$. Denote
	\[ \breve M = M\setminus \cup_{i=1}^k \Sigma_i. \]
	Tracking the homology exact sequence
	\[H_2(\partial \breve M; \ZZ)\xrightarrow{i_*} H_2(\breve M; \ZZ) \xrightarrow{j_*} H_2(\breve M, \partial \breve M; \ZZ) \xrightarrow{\partial} H_1(\partial \breve M; \ZZ),\]
	we find $\ker j_* = \operatorname{img} i_* = H_2(\breve M; \ZZ)$ (due to Lemma \ref{lemm:homology.nd.stable}), so $j_* = 0$, so $H_2(\breve M,\partial \breve M; \ZZ)$ injects into $H_1(\partial \breve M; \ZZ)$. By Lemma \ref{lemm:2d.metric.modification}, $\partial \breve M$ consists of $2$-spheres, so $H_1(\partial \breve M; \ZZ) = 0$ and thus $H_2(\breve M,\partial \breve M; \ZZ)=0$ by the exact sequence. By Lefschetz duality, $H^1(\breve M; \ZZ) = 0$. The Mayer--Vietoris sequence then implies $H^1(M'; \ZZ) = 0$, and thus $b_1(M') = 0$ by the universal coefficient theorem.
\end{proof}

\section{A Digression: Homotopy decomposition in 3D} \label{sec:homotopy.3d}

This section can be skipped at first reading. We digress to present a homotopy-based refinement of Theorem \ref{theo:homology.3d}, show how it can directly imply a certain homotopy version of Theorem \ref{theo:full.3d}, and discuss why we find this of interest. 

\begin{theo} \label{theo:homotopy.3d}
	Every closed, oriented, topologically PSC $3$-manifold can be obtained from a closed, oriented, topologically PSC $3$-manifold with vanishing $\pi_2$ (on each connected component) by performing $0$-surgeries.
\end{theo}

Instead of representing $H_2(M; \ZZ)$ by stable minimal $2$-spheres as in Section \ref{sec:homology.3d}, we seek to represent $\pi_2(M)$. The key ingredient is the Meeks--Yau proof of the embedded sphere theorem using minimal surfaces, which will be used in lieu of Lemma \ref{lemm:homology.nd.stable}.

\begin{lemm}[{\cite[Theorem 7]{MeeksYau:loop.sphere}} as Minimal sphere preparation lemma] \label{lemm:meeks.yau.sphere}
	Let $(M, g)$ be a closed, connected $3$-manifold. There exist conformal maps $\{ f_i : \SS^2 \to M \}_{i=1}^k$ with pairwise disjoint images that generate $\pi_2(M)$ as a $\pi_1(M)$-module. Each $f_i$ has least area in its homotopy class and is either a conformal embedding of an $\SS^2$ or a $2:1$ conformal covering of an embedded $\RR \PP^2$.
\end{lemm}

Of course, the embeddings or immersions above are stable minimal spheres. (The statement of \cite[Theorem 7]{MeeksYau:loop.sphere} doesn't explicitly mention the pairwise disjointness, but see the comment in the proof of \cite[p. 480, Assertion 1]{MeeksYau:loop.sphere}.) 

\begin{proof}[Proof of Theorem \ref{theo:homotopy.3d}]
	Without loss of generality, our initial manifold $M$ is connected. Endow $M$ with a PSC metric $g$. Let $\{ f_i : \SS^2 \to M \}_{i=1}^k$ be the maps given by Lemma \ref{lemm:meeks.yau.sphere}. The proof is the same as that of Theorem \ref{theo:homology.3d}, except for the fact that some images $\Sigma_i := f_i(\SS^2)$ might be doubly covered embedded one-sided $\RR \PP^2$'s. We apply Lemma \ref{lemm:2d.metric.modification} (see Remark \ref{rema:2d.metric.modification.onesided}) and modify $g$ locally near each $\Sigma_i$, arranging for a new PSC metric $\tilde g$ that, in the distance-$2$ tubular neighborhood of $\Sigma_i$ induces product metrics isometric to $(\SS^2, \varrho_{\SS^2}) \times (-2, 2)$ near $\Sigma_i$, where $\varrho_{\SS^2}$ is round, and all is taken modulo the $\ZZ_2$ action $(x, t) \mapsto (-x, -t)$ when $f_i$ is a double covering of $\Sigma_i$. 
	
	We excise $(M, \tilde g)$ as before, except for all $i$ corresponding with $\Sigma_i \approx \RR \PP^2$, we only use a single PSC $3$-ball and include an $\RR \PP^3$ in the new manifold $M'$. This $M'$ can be obtained from $M$ by $2$-surgeries, and thus $M$ can be obtained from $M'$ by $0$-surgeries. 
	
	It remains to verify that $\pi_2(M') = 0$. Without loss of generality, we may assume $M$ to be connected. Denote by $\pi : \tilde M \to M$ its universal covering and let $\tilde \Sigma := \pi^{-1}(\cup_{i=1}^k \Sigma_i)$. The boundary components of $\tilde M \setminus \tilde \Sigma$ are all $2$-spheres.
	
	\begin{claim} \label{clai:homotopy.3d.claim.components.sc}
		All components of $\tilde M \setminus \tilde \Sigma$ are simply connected. 
	\end{claim}

	\begin{proof}
			Let $\mathcal{E}_{1}$ be the closure of any component of $\tilde M \setminus \tilde \Sigma$. We inductively construct a monotone increasing exhaustion $\{ \mathcal{E}_{k} \}_{k=1,2,\ldots}$ of $\tilde M$ by taking $\mathcal{E}_{k+1} := \mathcal{E}_{k} \cup \mathcal{C}_{k}$ where $\mathcal{C}_{k}$ is the closure of any component of $\tilde M \setminus \tilde \Sigma$ that shares a boundary component with $\mathcal{E}_{k}$. At each step $\mathcal{E}_{k}$ is connected by construction. Moreover, $\partial \mathcal{E}_{k} \cap \partial \mathcal{C}_{k}$ has at most one component (otherwise one could construct a closed loop intersecting a component of $\partial \mathcal{E}_{k} \cap \partial \mathcal{C}_{k}$ in exactly one point, in contradiction to $\pi_1(\tilde M) = 0$). As a result, by the Seifert--Van Kampen theorem and the fact that all boundary components are $2$-spheres, the map $\pi_1(\mathcal{E}_{i})\to \pi_1(\mathcal{E}_{i+1})$ induced by the inclusion is injective. But $\pi_1(\tilde M) = 0$ is the direct limit of $\pi_1(\mathcal{E}_{i})$. Therefore, $\pi_1(\mathcal{E}_{1}) = 0$.
	\end{proof}
	
	Denote by $\tilde M'$ the manifold obtained by gluing one $3$-ball per boundary $2$-sphere of $\tilde M \setminus \tilde \Sigma$ and one $3$-sphere per component of $\pi^{-1}(\Sigma_i)$ for any $\Sigma_i \approx \RR \PP^2$. By construction:
	\begin{equation} \label{eq:mtildeprime.covers}
		\tilde M' \text{ covers } M'.
	\end{equation}
	By Seifert--Van Kampen and Claim \ref{clai:homotopy.3d.claim.components.sc}:
	\begin{equation} \label{eq:pi1.tildemprime}
		\pi_1(\tilde M') = 0.
	\end{equation}	
	Meanwhile, the argument in the proof of Theorem \ref{theo:homology.3d} gives:
	\begin{equation} \label{eq:h2.tildemprime}
		H_2(\tilde M'; \ZZ) = 0.
	\end{equation}		
	At this point, it follows that
	\[ \pi_2(M') \cong \pi_2(\tilde M') = 0; \]
	the isomorphism follows from \eqref{eq:mtildeprime.covers}, and the vanishing follows from \eqref{eq:pi1.tildemprime}, \eqref{eq:h2.tildemprime}, and the Hurewicz theorem. This completes the proof.
\end{proof}

Theorem \ref{theo:homotopy.3d} implies a certain homotopy version of Theorem \ref{theo:full.3d}:
\begin{coro} \label{coro:homotopy.3d}
	Every closed, oriented, topologically PSC $3$-manifold can be obtained from a closed, oriented, topologically PSC $3$-manifold with components covered by homotopy $3$-spheres, by performing $0$-surgeries.
\end{coro}

\begin{proof}
	Let $M'$ be the manifold obtained from Theorem \ref{theo:homotopy.3d}. Without loss of generality, $M'$ is connected. It suffices to show that $\pi_1(M')$ is finite. Suppose that $\pi_1(M')$ were infinite. Then the universal covering $\tilde M'$ of $M'$ would be noncompact, so $H_3(\tilde M'; \ZZ) = 0$. On the other hand, $\pi_2(\tilde M') \cong \pi_2(M') = 0$. Then, $\pi_3(M') \cong \pi_3(\tilde M') \cong H_3(\tilde M'; \ZZ) = 0$, the second isomorphism being Hurewicz's. This strategy iterates to give $\pi_k(M') = 0$ for all $k \geq 2$, and thus $M'$ is aspherical. This violates the non-asphericity result for 3D PSC: by Schoen--Yau and Gromov--Lawson (see \cite[Theorem E]{GromovLawson:summary}, \cite{SchoenYau:4d}), closed aspherical $3$-manifolds are not topologically PSC. 
\end{proof}

It is interesting to compare Corollary \ref{coro:homotopy.3d} with the Schoen--Yau and Gromov--Lawson approach to obtaining a homotopy version of Theorem \ref{theo:full.3d}:
\begin{enumerate}
	\item[(a)] By the Kneser--Milnor prime decomposition of $3$-manifolds \cite{Milnor:decomposition}, and further decomposing all prime $\SS^2 \times \SS^1$'s as $0$-surgeries on $\SS^3$'s, we see that any closed, oriented $3$-manifold $M$ can be obtained from:
	\begin{itemize}
		\item aspherical $3$-manifolds, and/or
		\item manifolds covered by homotopy $3$-spheres
	\end{itemize}
	by performing only $0$-surgeries. 
	\item[(b)] One knows \cite[Theorem E]{GromovLawson:summary} (see also \cite{SchoenYau:4d}) that aspherical $3$-manifolds \textit{cannot} occur as prime summands of PSC $3$-manifolds. This rules out the first source of summands in (a). Thus, any closed oriented $3$-manifold $M$ can be obtained from a closed, oriented $3$-manifold with components covered by homotopy $3$-spheres, by performing $0$-surgeries.
\end{enumerate}

In this approach, the topological decomposition in step (a) prevents step (b) from deciding whether the decomposed pieces, which are all covered by homotopy $3$-spheres, are themselves topologically PSC. (Of course, this follows from Perelman's resolution of the elliptization conjecture.) In our approach, step (a) was Theorem \ref{theo:homotopy.3d}, a geometric PSC surgery result that allows us to directly guarantee that the decomposed pieces are, too, topologically PSC. Step (b) is essentially unchanged for us.

A note about higher dimensions. Step (b) of the program above was recently carried out for $4$-(and $5$-)manifolds by the second author together with Otis Chodosh \cite{ChodoshLi:aspherical} (see also \cite{Gromov:aspherical, ChodoshLiLiokumovich:psc}). Step (a) remains a challenge. There are no suitable topological decompositions to perfectly replace step (a) in 4D, and we instead also proceed with a homological decomposition obtained via geometric measure theory.

\section{Main Theorem: Homology decomposition in 4D} \label{sec:homology.4d}

We now extend the strategy of Section \ref{sec:homology.3d} to 4D. We capture $H_3(M; \ZZ)$ by a stable minimal hypersurface $\Sigma$, using Lemma \ref{lemm:homology.nd.stable}. In the current higher dimensional setting we need to use Lemma \ref{lemm:3d.metric.modification} (instead of the two-dimensional Lemma \ref{lemm:2d.metric.modification}) to reduce to the product case. We state it below but defer its proof to Appendix \ref{sec:technical.lemmas}. 

\begin{lemm} \label{lemm:3d.metric.modification}
	Let $\Sigma$ be a two-sided, closed, embedded, stable, minimal hypersurface inside an oriented PSC $4$-manifold $(M, g)$. Then:
	\begin{enumerate}
		\item[(a)] $\Sigma$ must be topologically PSC and thus as in Theorem \ref{theo:full.3d}.
		\item[(b)] Given any auxiliary PSC metric $\sigma$ on $\Sigma$, there exists a new PSC metric $\tilde g$ on $M$, which:
			\begin{itemize}
				\item is isometric to a product cylinder $(\Sigma, \sigma) \times (-2, 2)$ in the distance-$2$ tubular neighborhood of $\Sigma$, and
				\item coincides with $g$ outside a larger tubular neighborhood of $\Sigma$.
			\end{itemize}
	\end{enumerate}
\end{lemm}

We will apply this lemma with various choices of $\sigma$, one being:

\begin{defi}[Model metric] \label{defi:3d.model.metric}
	Let $\Sigma$ be obtained by performing $0$-surgeries on spherical space forms. A PSC metric on $\Sigma$ is called a model metric if the neck corresponding to each $0$-surgery contains an isometric copy of $(\SS^2, \varrho) \times (-2, 2)$ for some size round metric $\varrho$ on $\SS^2$.
\end{defi}

\begin{proof}[Proof of Theorem \ref{theo:homology.4d}]
	Without loss of generality, $M$ is connected. Fix a PSC metric $g$ on $M$, let $\Sigma$ be as in Lemma \ref{lemm:homology.nd.stable}. We will stray slightly from the notation in the statement of Theorem \ref{theo:homology.4d} below: $M'$ will not denote the ultimate decomposition, but only an intermediate one, and the ultimate decomposition will be denoted $M''$. 
	
	\textbf{Step I}. In this step, we assume that $\Sigma$ involves at least one $0$-surgery, otherwise we proceed to \textbf{Step II} by setting $(M', g') := (M, g)$ and $\Sigma' := \Sigma$.
	
	By Lemma \ref{lemm:3d.metric.modification} we can modify the metric $g$ of $M$ locally near $\Sigma$ and arrange for a new PSC metric $\tilde g$ on $M$ that induces a product metric $(\Sigma, \sigma) \times (-2, 2)$ in the distance-$2$ tubular neighborhood of $\Sigma$, where $\sigma$ is a model metric (Definition \ref{defi:3d.model.metric}).
	
	Now consider the neck regions $\{ N_i \}_{i=1}^\ell$ of $\Sigma$ where, due to the model metric structure, $(\Sigma, \sigma)$ restricts on $N_i$ to an isometric copy of $(\SS^2, \varrho_i)$, where $\varrho_i$ is a round metric on $\SS^2$ of radius $\eps_i$. By construction of $\tilde g$, the distance-$2$ tubular neighborhood of $N_i$ in $(M, \tilde g)$ is isometric to $(\SS^2, \varrho_i) \times (-2, 2) \times (-2, 2)$. Let $U_i$ be the interior of a smoothing of the domain $\SS^2 \times [-1, 1] \times [-1, 1]$ in these coordinates, where the smoothing only takes place outside $\SS^2 \times [-\tfrac12, \tfrac12] \times [-1, 1]$ and $\SS^2 \times [-1,1] \times [-\tfrac12, \tfrac12]$. Note that $U_i \approx \SS^2 \times B^2$. For small $\eps_i$, we can construct a PSC metric $\eta_i$ on $V_i := B^3 \times \SS^1$ which matches smoothly with $U_i$ on their respective boundaries ($\approx \SS^2 \times \SS^1$). Then,
	\[ (M', g') := (M \setminus \cup_{i=1}^\ell U_i, \tilde g) \cup \big( \cup_{i=1}^\ell (V_i, \eta_i) \big) \]
	with the obvious boundary identifications, is a PSC 4-manifold obtained from $M$ by $2$-surgeries. Thus, $M$ can be obtained from $M'$ by $1$-surgeries. The surgeries above can be performed so that $\Sigma$ gets replaced by a hypersurface $\Sigma'$, whose components are spherical space forms.
	Moreover, by choosing $(V_i, \eta_i)$ to be a local product on the $\SS^1$ factor, we may assume that the metric $g'$ is locally a product near $\Sigma'$ and ensure that $\Sigma'$ is again stable.
	
	Note that the components of $M' \setminus \Sigma'$ arise from the components of $M \setminus \Sigma$ by attaching copies of $B^3 \times [-1,1]$ along $\SS^2 \times [-1,1]$ to its boundary components. This shows that the surjectivity of the natural map 
	$H_3(\partial (M' \setminus \Sigma'); \ZZ) \twoheadrightarrow H_3(M' \setminus \Sigma'; \ZZ)$,
	initially guaranteed in Lemma \ref{lemm:homology.nd.stable} 
	for $H_3(\partial (M \setminus \Sigma); \ZZ) \twoheadrightarrow H_3(M \setminus \Sigma; \ZZ)$, 
	is maintained. 
	
	Finally, since we're performing $2$-surgeries, we have $b_2(M') \leq b_2(M)$. (See Lemma \ref{lemm:1.surg.b1.b2}, where the roles of $M$ and $M'$ are reversed.)
		
	\textbf{Step II}. In this step we work with $(M', g')$ and the components $\{ \Sigma_i' \}_{i=1}^{k'}$ of $\Sigma'$, each of which is, by construction, a spherical space form, i.e., $\Sigma_i' \approx \SS^3 / \Gamma_i$, for finite subgroups $\Gamma_i$ of $SO(4)$ acting freely on $\SS^3$.
	
	We invoke Lemma \ref{lemm:3d.metric.modification} again, except we modify the metric $g'$ of $M'$ locally near each $\Sigma_i'$ and arrange for a new PSC metric $\tilde g'$ that induces a product metric $(\Sigma_i', \sigma_i') \times (-2, 2)$ in the distance-$2$ tubular neighborhood of $\Sigma_i'$, where $\sigma_i'$ is a round metric on $\Sigma_i'$. For convenience, we denote by $A \subseteq \{ 1, 2, \ldots, k' \}$ the set of $i$'s for which $\Sigma_i' \approx \SS^3$, and by $B \subseteq \{ 1, 2, \ldots, k' \}$ the remaining $i$'s.
	
	For each $i \in A$, we can excise the distance-$1$ tubular neighborhood $U_i'$ of $\Sigma_i'$ and smoothly replace it with two PSC 4-balls $\{ (B^\alpha_i, \theta_i^\alpha) \}_{\alpha=1}^2$. Note that this is a $3$-surgery, and is therefore undone with a $0$-surgery.
	
	For each $i \in B$, we can still excise the distance-$1$ tubular neighborhood $U_i'$ of $\Sigma_i'$, but now we have to smoothly glue in two PSC 4-orbifold-balls $\{ (\hat B^\alpha_i, \hat \theta_i^\alpha) \}_{\alpha=1}^2$ whose orbifold singularity is modeled on on $\RR^4 / \Gamma_i$. Note that this is an orbifold $3$-surgery, and is therefore undone by a $0$-surgery at orbifold points.
	
	The ultimate space we end up with is
	\[ (M'', g'') := (M' \setminus \cup_{i=1}^k U_i', \tilde g') \cup \big( \cup_{i \in A} \cup_{\alpha=1,2} (B^\alpha_i, \theta^\alpha_i) \big) \cup \big( \cup_{i \in B} \cup_{\alpha=1,2} (\hat B^\alpha_i, \hat \theta^\alpha_i) \big). \]
	As explained, $M$ can be obtained from $M''$ by performing $0$-surgeries (possibly on the orbifold points) and then $1$-surgeries (on the smooth part).
	
	We now prove $b_1(M'') = 0$. Set $\breve M := M' \setminus \Sigma'$. As in the proof of Theorem~\ref{theo:homology.3d}, we track the homology exact sequence
	\[ H_3(\partial \breve M; \ZZ)\xrightarrow{i_*} H_3(\breve M; \ZZ) \xrightarrow{j_*} H_3(\breve M, \partial \breve M; \ZZ) \xrightarrow{\partial} H_2(\partial \breve M; \ZZ).\]
	We have $H_2(\partial \breve M; \ZZ) = 0$ since the components $\Sigma_i'$ of $\partial \breve M$ are spherical space forms. As before, $i_*$ is guaranteed to be surjective, so $H_3(\breve M, \partial \breve M; \ZZ) = 0$, so $H^1(\breve M; \ZZ) = 0$ by Lefschetz duality, and $b_1(M'') = 0$ from Mayer--Vietoris. 
	
	Finally, $b_2(M'') = b_2(M') \leq b_2(M)$. The inequality follows from \textbf{Step I} and the equality from the $3$-surgeries on $M'$ to get $M''$. The latter can be verified applying Mayer--Vietoris twice to get $b_2(M'') = b_2(\breve M) = b_2(M')$: once on the open cover of $M'$ by $\breve M$ and $\cup_{i=1}^k U_i'$, and once on the open cover of $M''$ by $\breve M$ and $(\cup_{i,\alpha} B_i^\alpha) \cup (\cup_{i,\alpha} \hat B_i^\alpha)$.
\end{proof}

\begin{rema} \label{rema:homology.4d.orbifold}
	The orbifold singularities in the decomposition are due to the spherical space forms appearing in $\Sigma'$ in \textbf{Step II} above. 
\end{rema}

\begin{rema} \label{rema:spin.4d.conclusion}
	Our proof of Theorem \ref{theo:homology.4d} preserves the spin condition; i.e., if $M$ is spin, then the ultimate orbifold $M'$ is spin too, in the sense that its regular part $M''_{\text{reg}}$ is spin. In Step I we can endow each $B^3 \times \SS^1$ we are gluing in to replace $U_i$ with a spin structure to induce the same spin structure on $\SS^2 \times \SS^1$ that $\partial U_i$ did; one can see this, e.g., from that $\pi_1(\SS^2 \times \SS^1) \to \pi_1(B^3 \times \SS^1)$ is an isomorphism and $\pi_2(B^3 \times \SS^1) = 0$. Likewise, in Step II we can similarly endow each $B_i^\alpha$ or $\hat B_i^\alpha$ with a spin structure to induce the same spin structure on each $\SS^3 / \Gamma_i$ that $\Sigma_i'$ did from either side.
\end{rema}

\section{Proof of Theorem \ref{theo:s3g.s1}} \label{sec:s3g.s1}

Suppose, for the sake of contradiction, that $M = (\SS^3 / \Gamma) \times \SS^1$ can be obtained by manifold $0$ and $1$-surgeries from a $4$-manifold $M'$ with $b_1(M')=0$. The key features of $M$ are that 
\begin{equation} \label{eq:s3g.s1.pi1} 
	\pi_1(M) = \Gamma \times \ZZ \text{ with } \Gamma \leq SO(4) \text{ finite, cyclic, nontrivial,}
\end{equation}
\begin{equation} \label{eq:s3g.s1.b1.b2}
	b_1(M) = 1, \; b_2(M) = 0.
\end{equation}
Since $1<4/2$, all the $0$ and $1$-surgeries commute. So first take the connected sum of all connected components of $M'$, and call it $M''$. Note that $b_1(M'') = 0$. Perform all remaining $0$-surgeries, so that $M''$ turns into $M''\# k(\SS^3\times \SS^1)$ for some integer $k \geq 0$. Then, $M$ is obtained by performing $1$-surgeries on $M''\# k(\SS^3\times \SS^1)$, which has
\[ b_1(M'' \# k(\SS^3 \times \SS^1)) = k. \]
Combining Lemma \ref{lemm:1.surg.b1.b2} with \eqref{eq:s3g.s1.b1.b2}, we see that the integer $k$ above is such that $k \geq 1$ and that exactly $(k-1)$ $1$-surgeries were performed. 

On the level of fundamental groups, since $M$ is obtained from $M'' \# k(\SS^3 \times \SS^1)$ by performing $(k-1)$ $1$-surgeries, \eqref{eq:s3g.s1.pi1}  yields
\begin{equation*} 
	\Gamma \times \ZZ \cong \left\langle G * \ZZ^{*k} \bigg\vert r_1,\cdots,r_{k-1}\right\rangle,
\end{equation*}
where $G = \pi_1(M'')$, $\ZZ^{*k} = \ZZ * \cdots * \ZZ$ ($k$ times) and $r_1, \ldots, r_{k-1}$ represent the relations (possibly trivial) introduced by the $1$-surgeries. We are led to a contradiction from this presentation of $\Gamma \times \ZZ$ and that the abelianization of $G$ has rank $b_1(M'') = 0$ via the following group theoretic lemma.

\begin{lemm}
	Suppose that $\Gamma$ is a nontrivial finite cyclic group, and that G is a group whose abelianization has rank 0. Then $\Gamma \times \ZZ$ cannot be expressed as
	\begin{equation} \label{eq:s3g.s1.iso}
		\Gamma \times \ZZ \cong \left\langle G * \ZZ^{*k} \bigg\vert r_1,\cdots,r_{k-1}\right\rangle.
	\end{equation}
\end{lemm}
\begin{proof}
If $N$ is the minimal normal subgroup of $G * \ZZ^{*k}$ containing $\langle r_1,\cdots,r_{k-1}\rangle$, then, by \eqref{eq:s3g.s1.iso}, 
\begin{equation}\label{eq.group.1}
	 \Gamma\times \ZZ \cong \quot{G* \ZZ^{*k}}{N}.
\end{equation}
Consider the standard embedding $i: G\to G*\ZZ^{*k}$, and consider the normal subgroup
\[ G' = i^{-1}(N) \trianglelefteq G. \]
We will reduce to the case $G' = \{ 1 \}$. If $N'$ denotes the minimal normal subgroup in $G*\ZZ^{*k}$ containing $i(G')$, then
\begin{equation*}
	\quot{G*\ZZ^{*k}}{N}\cong \quot{\textstyle \left(\quot{G*\ZZ^{*k}}{N'}\right)}{\textstyle\left(\quot{N}{N'}\right)}
\end{equation*}
If we denote by $\tilde r_i$ the image of $r_i$ under the map $G*\ZZ^{*k} \to (\quot{G}{G'})*\ZZ^{*k}$, and we denote by $\tilde N$ the minimal normal subgroup of $(\quot{G}{G'}) * \ZZ^{*k}$ containing $\langle \tilde r_1,\cdots,\tilde r_{k-1}\rangle$, then the rightmost group above is
\[ \cong \quot{\left(\quot{G}{G'}\right)*\ZZ^{*k}}{ \tilde N }.\]
Also, under the standard embedding $\tilde i: \quot{G}{G'}\to(\quot{G}{G'})*\ZZ^{*k}$, we have
\[(\tilde i)^{-1}(\tilde N)=\{1\}.\]
Thus, by replacing $G$ with $\quot{G}{G'}$ and $r_i$ with $\tilde r_i$, we may assume that we had
\[ i^{-1}(N) = \{ 1 \}, \]
all along as desired. Note that this reduction preserves the property that the rank of the abelianization equals zero. Then, 
\[G\xrightarrow{i}G*\ZZ^{*k} \to \quot{G*\ZZ^{*k}}{N} \cong \Gamma \times \ZZ,\]
(the last isomorphism is \eqref{eq.group.1}) is injective. Therefore $G$ is abelian since $\Gamma \times \ZZ$ is. The rank of the abelianization of $G$, and thus of $G$, is assumed to be $0$, so $G$ is finite. Since $\ZZ$ has no torsion, $G$ must inject into $\Gamma \times \{ 0 \} \cong \Gamma$, which is assumed to be finite and cyclic, so
\[G \cong \ZZ_{m},\quad m\ge 1.\]
Therefore, 
\begin{equation} \label{eq.group.iso}
	\Gamma\times \ZZ \cong \left\langle \ZZ_m*\ZZ^{*k} \bigg\vert r_1,\cdots,r_{k-1}\right\rangle \cong \left\langle \ZZ*\ZZ^{*k} \bigg\vert r_0,r_1,\cdots,r_{k-1}\right\rangle,
\end{equation}
where $r_0$ is a word that gives order $m$ to the generator of the first $\ZZ$ factor. Thus, for every $n\ge 1$, 
\begin{equation} \label{eq.group.2}
	\Gamma\times \ZZ_n \cong \left\langle \ZZ*\ZZ^{*k} \bigg\vert r_0,r_1,\cdots,r_{k-1},r_k\right\rangle,
\end{equation}
with $r_k$ being a word that reduces to $(0, n)$ in \eqref{eq.group.iso}.

We show that \eqref{eq.group.2} must fail for \textit{some} $n$ by comparing the Schur Multiplier (denoted by $M(\cdot)$) of both sides. Indeed, it follows from \eqref{eq.group.2} that $\Gamma \times \ZZ_n$ is a finitely represented group with the same number of generators and relations (a group of zero deficiency), so 
\[ M(\Gamma \times \ZZ_n)=\{1\} \]
by \cite[Lemma 1.2]{Zerodeficiency}. On the other hand, if we choose $n$ to be equal to the order of $\Gamma$, so $\Gamma \cong \mathbf{Z}_n$, then, invoking \cite[Theorem 2.1]{SchurMulti} yields
\[M(\Gamma \times \ZZ_n) = M(\ZZ_n\times \ZZ_n) = M(\ZZ_n)\times M(\ZZ_n)\times (\ZZ_n\otimes \ZZ_n)=\ZZ_n.\]
This is a contradiction since $n \geq 2$. 
\end{proof}

\appendix

\section{Technical lemmas} \label{sec:technical.lemmas}

We first prove Lemmas \ref{lemm:2d.metric.modification}, \ref{lemm:3d.metric.modification}. Our proof relies on the flexibility of two-sided stable minimal hypersurfaces in $3$- and $4$-manifolds due to the second and third authors \cite{LiMantoulidis:metrics.lambda1}.\footnote{Another proof relies on the conformal cobordism theory of Akutagawa--Botvinnik \cite{AkutagawaBotvinnik} and the known path-connectedness of the space of PSC metrics on topologically PSC manifolds due to Weyl \cite{Weyl} in 2D and the first author and Bruce Kleiner \cite{BamlerKleiner:contractibility} in 3D. We thank Demetre Kazaras for bringing this alternative and particularly \cite{AkutagawaBotvinnik} to our attention.} We need one additional piece of notation and an auxiliary lemma. For any closed connected manifold $\Sigma$, set
\[ \sM(\Sigma) := \{ \sigma \in \operatorname{Met}(\Sigma) : -\Delta_\sigma + \tfrac12 R_\sigma \text{ has positive principal eigenvalue} \}; \]
here, $R_\sigma$ denotes the scalar curvature of $\sigma$. This space is denoted $\sM^{>0}_{1/2}(\Sigma)$ in \cite{LiMantoulidis:metrics.lambda1}.

\begin{lemm} \label{lemm:collar}
	Suppose $(\sigma_t)_{t \in [0,1]}$ is a smooth path in $\sM(\Sigma)$ with $\sigma_0' \equiv \sigma_1' \equiv 0$. There exists a smooth $u : \Sigma \times [0,1] \to (0, \infty)$ so that the metric $h = \sigma_t + u^2 dt^2$ on $\Sigma \times [0, 1]$ has the following properties:
	\begin{enumerate}
		\item[(a)] $\Sigma \times \{0\}$ and $\Sigma \times \{1\}$ are totally geodesic;
		\item[(b)] $R_h > 0$ everywhere.
	\end{enumerate}
\end{lemm}
\begin{proof}
	We make use of the curvature formulas from \cite[Lemma A.1]{LiMantoulidis:metrics.lambda1}. Right away we note that (a) is a consequence of $\sigma_0' \equiv \sigma_1' \equiv 0$, no matter what $u$ is. For (b), the key observation is that if $\tilde u : \Sigma \times [0,1] \to (0, \infty)$ is held fixed and $A > 0$ is a constant, both to be determined, then
	\begin{equation} \label{eq:collar.scalar.curv}
		R_{\sigma_t + A^2 \tilde u^2 \, dt^2} = 2 \tilde u^{-1} (\Delta_{\sigma_t} \tilde u + \tfrac12 R_{\sigma_t} \tilde u) + O(A^{-2}) \text{ as } A \to \infty.
	\end{equation}
	Take each $\tilde u(\cdot, t) : \Sigma \to (0, \infty)$ to be a positive principal eigenfunction of $-\Delta_{\sigma_t} + \tfrac12 R_{\sigma_t}$. This can be done smoothly over $\Sigma \times [0,1]$ because the principal eigenvalue is simple; see \cite[Lemma A.1]{MantoulidisSchoen:bartnik}. With this $\tilde u$, the first term on the right hand side of \eqref{eq:collar.scalar.curv} is bounded below by the minimum principal eigenvalue over $t$, thus uniformly positive since $\sigma_t \in \sM(\Sigma)$ for all $t$. Now, (b) follows with $u := A \tilde u$ and $A \gg 1$.
\end{proof}

\begin{proof}[Proof of Lemma \ref{lemm:2d.metric.modification}]
	Conclusion (a) is a well-known consequence of \cite{SchoenYau:incompressible}. 
	
	To arrange conclusion (b), we first cut $(M, g)$ along $\Sigma$. This introduces two boundary components isometric to $(\Sigma, \sigma_0)$. We will join them together using a PSC cylinder constructed using Lemma \ref{lemm:collar}, and its reflection. 
	
	To that end, note that $\sigma_0 \in \mathcal{M}(\Sigma)$ by \cite[Lemma C.6]{LiMantoulidis:metrics.lambda1}. The auxiliary metric $\sigma$ on $\Sigma$ also satisfies $\sigma \in \mathcal{M}(\Sigma)$ because it has positive scalar curvature. The existence of a smooth path $(\sigma_t)_{t \in [0,1]} \subset \sM(\Sigma)$ with $\sigma_1 = \sigma$ is guaranteed by \cite[Proposition 1.1]{MantoulidisSchoen:bartnik}; the condition $\sigma_0' \equiv \sigma_1' \equiv 0$ is trivially arranged by a time reparametrization. Then Lemma \ref{lemm:collar} yields a PSC metric $h$ on $\Sigma \times [0, 1]$ that induces $\sigma_t$ on $\Sigma \times \{t\}$ for all $t \in [0,1]$, with both boundary components totally geodesic. We then extend $h$ to a metric on $\Sigma \times [0,4]$ by taking it to be a product metric $(\Sigma, \sigma) \times (1, 4]$.
	
	Along each boundary component of $(M, g) \setminus \Sigma$, we glue in a copy of $(\Sigma \times [0, 4], h)$, identifying the boundary $\Sigma$ with $\Sigma \times \{ 0\}$, and the two copies of $\Sigma \times \{4\}$ with each other, suitably reversing orientations. This yields a new manifold diffeomorphic to $M$, with a Lipschitz metric $\hat g$ that satisfies all desired conclusions of the lemma, except it is only smooth away from the various copies of $\Sigma \times \{0, 1\}$. Since these hypersurfaces are minimal from both sides in $(M, \hat g)$, they can be smoothed out locally while preserving PSC. One can do this using \cite{Miao:pmt.corners}, which is easily localized using cut-off functions in view of the strict positivity of scalar curvature in our case; see \cite[Lemma 7.1]{LiMantoulidis:metrics.lambda1} for details.
\end{proof}

\begin{proof}[Proof of Lemma \ref{lemm:3d.metric.modification}]
	The proof of Lemma \ref{lemm:2d.metric.modification} carries through verbatim, except conclusion (a) is due to \cite{SchoenYau:structure.psc} and one needs to invoke \cite[Proposition 3.1]{LiMantoulidis:metrics.lambda1} rather than \cite[Proposition 1.1]{MantoulidisSchoen:bartnik} when proving (b).
\end{proof}

\begin{proof}[Proof of Lemma \ref{lemm:homology.nd.stable}]
We offer a generalization of the method \cite[Lemma 19]{ChodoshLi:aspherical}.

We proceed by induction.
Suppose that we have constructed a sequence of pairwise disjoint, two-sided, stable hyperminimal surfaces $\Sigma_1, \ldots, \Sigma_k \subset M$ such that $\breve M_k := M \setminus \cup_{j=1}^k \Sigma_j$ is connected.
If the map 
\begin{equation*} \label{eq_mapHn1}
i_k : H_{n-1} (\partial \breve M_k; \ZZ) \to H_{n-1}(\breve M_k; \ZZ)
\end{equation*}
is surjective, we are done; set $\Sigma := \cup_{j=1}^k \Sigma_j$ in the statement of the lemma. So let us assume it is not. Using geometric measure theory \cite{Federer:GMT} (this is the source of the dimensional restriction $n \leq 7$) we can find a closed, connected, two-sided, stable minimal surface $\Sigma_{k+1} \subset \breve M_{k}$ such that $\breve M_{k+1} := \breve M_k \setminus \Sigma_{k+1}$ is connected and such that
\begin{equation} \label{eq_Sigmaimgik}
[\Sigma_{k+1}] \not\in \operatorname{img} i_k.
\end{equation}
To see that this process terminates, we will show that the cokernels 
\[ \Gamma_k := H_{n-1}(\breve M_{k};\ZZ) / \operatorname{img} i_{k} \]
have strictly decreasing rank. Since $\Gamma_0 = H_{n-1}(M;\ZZ)$ is finitely generated, the process will have to terminate after finitely many steps.

So, it remains to prove the ranks decrease strictly. We apply the Mayer--Vietoris sequence to the open cover of $\breve M_k$ consisting of $\breve M_{k+1}$ and a tubular neighborhood of $\Sigma_{k+1}$. Writing $H_j (\Sigma_{k+1};\ZZ)$ and $H_{j}(\Sigma_{k+1} \times \{ \pm 1\};\ZZ)$ for the  homology groups of the tubular neighborhood of $\Sigma_{k+1}$ and for the intersection of both subsets, respectively:
\[ H_{n-1}(\Sigma_{k+1} \times \{ \pm 1 \} ;\ZZ) \to H_{n-1}(\breve M_{k+1};\ZZ) \oplus H_{n-1}(\Sigma_{k+1};\ZZ) \to H_{n-1} (\breve M_k;\ZZ) \]
This implies that the kernel of the natural map
\begin{equation} \label{eq_Hn1MM}
 H_{n-1}(\breve M_{k+1};\ZZ) \to H_{n-1} (\breve M_k;\ZZ), 
\end{equation}
which is induced by the inclusion map $\breve M_{k+1} 
\hookrightarrow \breve M_k$, is contained in $\operatorname{img} i_{k+1}$.
Thus the map \eqref{eq_Hn1MM} descends to an injection 
\begin{equation} \label{eq_Gammak_injection}
	\Gamma_{k+1} \hookrightarrow H_{n-1}(\breve M_{k};\ZZ) \big/ (\operatorname{img} i_{k} + \ZZ \cdot [\Sigma_{k+1}]) = \Gamma_k / (\ZZ \cdot [[\Sigma_{k+1}]]),
\end{equation}
where $[[\Sigma_{k+1}]]$ denotes the equivalence class of $[\Sigma_{k+1}]$ within $\Gamma_k$. 

Note that the $\Gamma_k$ are all torsion-free. Indeed, from the homology long exact sequence
\[ H_{n-1}(\partial \breve M_k; \ZZ)\xrightarrow{i_*} H_{n-1}(\breve M_k; \ZZ) \rightarrow H_{n-1}(\breve M_k, \partial \breve M_k; \ZZ) \xrightarrow{\partial} H_{n-2}(\partial \breve M; \ZZ)\]
we get $\Gamma_k \cong \ker \partial$. However, $H_{n-1}(\breve M_k, \partial \breve M_k; \ZZ) \cong H^1(\breve M_k; \ZZ)$ by Lefschetz duality. The latter is torsion-free by the universal coefficient theorem. Thus, $\Gamma_k$ is torsion-free. 

Since the $\Gamma_k$ are torsion free, $\rank \Gamma_{k+1} < \rank \Gamma_k$ follows from \eqref{eq_Sigmaimgik} and \eqref{eq_Gammak_injection}. 
\end{proof}

The following lemma is well-known to the experts but we include a proof for clarity:

\begin{lemm} \label{lemm:1.surg.b1.b2}
	Suppose $M$, $M'$ are closed, oriented $4$-manifolds and that $M$ is obtained from $M'$ by a $1$-surgery. Then $b_1(M) \leq b_1(M')$ and $b_2(M) \geq b_2(M')$. In fact, either
	\[ b_1(M) = b_1(M') - 1 \text{ and } b_2(M) = b_2(M') \]
or
	\[ b_1(M) = b_1(M') \text{ and } b_2(M) = b_2(M') + 2. \]
\end{lemm}
\begin{proof}
	Let $\breve M$ be the manifold obtained by removing a $\SS^1\times B^3$ from $M'$. The Mayer--Vietoris sequence implies that
	\[H_0(\SS^1\times \SS^2; \ZZ) \xrightarrow{((i_0)_*, (j_0)_*)} H_0(\breve M; \ZZ)\oplus H_0(\SS^1\times B^3; \ZZ) \rightarrow H_0(M'; \ZZ) \to \{1\}.\]
	For $H_0$, the map $((i_0)_*, (j_0)_*)$ is the diagonal map, so it is injective. Therefore, again from Mayer--Vietoris:
	\[H_1(\SS^1\times \SS^2; \ZZ) \xrightarrow{((i_1)_*, (j_1)_*)} H_1(\breve M; \ZZ)\oplus H_1(\SS^1\times B^3; \ZZ) \rightarrow H_1(M'; \ZZ)\rightarrow \{1\}. \]
	Observe that $j_1: \SS^1\times \SS^2\to \SS^1\times B^3$ induces an injection on $H_1$. Counting ranks in
	\[\quot{\left(H_1(\breve M; \ZZ)\oplus H_1(\SS^1\times B^3; \ZZ)\right)}{\imag ((i_1)_*, (j_1)_*))} \cong H_1(M'; \ZZ),\]
	we conclude that $b_1(\breve M)= b_1(M')$.

	Similarly, 
	\[H_1(\SS^2\times \SS^1; \ZZ) \xrightarrow{((i_1)_*, (k_1)_*))} H_1(\breve M; \ZZ)\oplus H_1(\SS^2\times B^2; \ZZ)\rightarrow H_1(M; \ZZ) \to \{1\}. \]
	Counting ranks, $b_1(\breve M) = b_1(M) + \rank \imag((i_1)_*, (k_1)_*) = b_1(M) +  \rank \imag (i_1)_*$. Combing these, we have that
	\begin{equation} \label{eq:1.surg.b1}
		b_1(M)=b_1(M') \text{ or } b_1(M)=b_1(M')-1,
	\end{equation}
	depending on whether $\rank \imag (i_1)_* = 0$ or $1$. 

	On the other hand, the Euler characteristics of $X, Y$ satisfy:
	\begin{align*}
		\chi(M') = \chi(\breve M) + \chi(\SS^1\times B^3) - \chi(\SS^1\times \SS^2),\\
		\chi(M) = \chi(\breve M) + \chi(\SS^2\times B^2) - \chi(\SS^1\times \SS^2).
	\end{align*}
	Therefore, $\chi(M)=\chi(M')+2$. The result follows in combination with \eqref{eq:1.surg.b1}.
\end{proof}

\bibliography{main} 
\bibliographystyle{amsalpha}

\end{document}